\documentclass[12pt]{amsart}
\usepackage{amsmath}
\usepackage{amsthm}
\usepackage{amssymb}
\usepackage{amsfonts,mathrsfs}
\usepackage{stmaryrd}
\usepackage{amsxtra}
\usepackage{epsfig}
\usepackage{verbatim}
\usepackage{color}

\theoremstyle{definition}
\newtheorem{theorem}[equation]{Theorem}
\newtheorem{proposition}[equation]{Proposition}
\newtheorem{lemma}[equation]{Lemma}

\newtheorem{definition}[equation]{Definition}
\theoremstyle{remark}
\newtheorem*{remark}{Remark}
\numberwithin{equation}{section}

\include{header}

\newcommand{\wB}{\mathbf{B}_{\lambda}}

\begin{document}

\title[Weighted Bergman Projections]{Weighted Bergman Projections and Kernels: $L^p$ Regularity and Zeros }
\author{Yunus E. Zeytuncu}
\date{}
\address{Department of Mathematics, Texas A\&M University, College Station, Texas 77843}
\email{yunus@math.ohio-state.edu \\ zeytuncu@math.tamu.edu}
\begin{abstract} We investigate $L^p$ regularity of weighted Bergman projections and zeros of weighted Bergman kernels for the weights that are radially symmetric and comparable to 1 on the unit disc.
\end{abstract}
\maketitle

\section{Introduction}

\subsection{Setup and Notation}

Let $\mathbb{D}$ denote the unit disc in $\mathbb{C}^1$ and let $\lambda(r)$ be a non-negative function on $[0,1)$ that is comparable to 1. 

\begin{definition}
A  function $\lambda(r)$ on $[0,1)$ is said to be comparable to $1$ and  denoted by $\lambda\sim 1$, if there exists $C>0$ such that $\frac{1}{C}\leq \lambda(r)\leq C$ for all $r\in[0,1)$.\end{definition}

We consider $\lambda$ as a radial weight on $\mathbb{D}$ by setting $\lambda(z)=\lambda(|z|)$. We denote the Lebesgue measure on $\mathbb{C}$ by $dA(z)$ and  the space of square integrable functions on $\mathbb{D}$ with respect to the measure $\lambda(z)dA(z)$ by $L^2(\lambda)$. This is a Hilbert space with the inner product and the norm defined by
\begin{align*}
\left<f,g\right>_{\lambda}=\int_{\mathbb{D}}f(z)\overline{g(z)}\lambda(z)dA(z)~\text{ and }~
||f||_{\lambda}^2=\int_{\mathbb{D}}|f(z)|^2\lambda(z)dA(z).\\
\end{align*}

The space of holomorphic functions that are in $L^2(\lambda)$ is denoted by $A^2(\lambda)$. The Bergman inequality (see the first page of \cite{Duren04}) shows that $A^2(\lambda)$ is a closed subspace of $L^2(\lambda)$. The orthogonal projection between these two spaces is called \textit{the weighted Bergman projection} and denoted by $\wB$, i.e.
\begin{equation}\label{projection}
\mathbf{B}_{\lambda}: L^2(\lambda) \to A^2(\lambda).
\end{equation}\\

It follows from the Riesz representation theorem that $\wB$ is an integral operator. The kernel is called \textit{the weighted Bergman kernel} and denoted by $B_{\lambda}(z,w)$, i.e. for any $f\in L^2(\lambda)$,
\begin{equation}\label{integraloperator}
\mathbf{B}_{\lambda}f(z)=\int_{\mathbb{D}} B_{\lambda}(z,w)f(w)\lambda(w)dA(w).
\end{equation}

For a radial weight $\lambda$ as above, the monomials $\{z^n\}_{n=0}^{\infty}$ form an
orthogonal basis for $A^2(\lambda)$ and after normalization the weighted Bergman kernel is
given by, $B_{\lambda}(z,w)=\sum_{n=0}^{\infty}\alpha_n(z\bar w)^n$
 where $\alpha_n=\frac{1}{2\pi\int_0^1r^{2n+1}\lambda(r)dr}.$\\

The general theory and details can be found in \cite{ForelliRudin} and \cite{Zhubook}.\\

\subsection{Questions and Results}

Both the weighted Bergman projection and the weighted Bergman kernel are canonical objects on the weighted space $(\mathbb{D},\lambda)$. It is a fundamental problem how the perturbations of $\lambda$ affect the analytic properties of these canonical objects. In this note, we are particularly interested in the following questions.\\
\begin{itemize}
\item[\textbf{I}.] For a given radial weight $\lambda$ on $\mathbb{D}$ as above, does the weighted kernel $B_{\lambda}(z,w)$ have zeros in $\mathbb{D}\times\mathbb{D}$?\\
\item[\textbf{II}.] For a given radial weight $\lambda$ on $\mathbb{D}$ as above, for which values of $p\in(1,\infty)$ is the weighted projection $\wB$ bounded from $L^p(\lambda)$ to $L^p(\lambda)$?\\
\end{itemize}
The answers for both of the questions are known for $\lambda(r)\equiv 1$. In this case, a direct computation gives that
\begin{equation}\label{kernel}
B_{1}(z,w)=\frac{1}{\pi}\sum_{n=0}^{\infty}(n+1)(z\bar w)^n=\frac{1}{\pi(1-z\bar{w})^2}.
\end{equation}
This closed form immediately tells that $B_1(z,w)$ never vanishes inside $\mathbb{D}\times \mathbb{D}$. Moreover, this explicit form and Schur's lemma prove that $\mathbf{B}_1$ is bounded from $L^p(1)$ to $L^p(1)$ for all $p\in(1,\infty)$.\\ 

For an arbitrary radial weight, in general, we do not have such a closed form for the kernel and it requires more work to answer questions above. In the next two sections, we prove the following two theorems that answer Questions \textbf{I} and \textbf{II}.\\

\begin{theorem}\label{zeros}
There exists a radial weight $\lambda$ on $\mathbb{D}$, comparable to $1$, such that the kernel function $B_{\lambda}(z,w)$ has zeros.\\
\end{theorem}

\begin{theorem}\label{one}
Suppose $\lambda$ is a non-negative radial weight on $\mathbb{D}$ that is comparable to $1$. 
Then the weighted Bergman projection $\mathbf{B}_{\lambda}$ is bounded from $L^p(\lambda)$ to $L^p(\lambda)$ for all $p\in(1,\infty)$.\\
\end{theorem}

In the sprit of the problem above, these two theorems say that a comparable perturbation of the radial weight on $\mathbb{D}$ might change the vanishing properties of the kernel but does not alter the $L^p$ mapping properties of the projection.\\

I thank J.D. McNeal, my advisor, for introducing me to this field and helping me with various points. I thank H.P. Boas for helpful remarks on an earlier version of this paper.

\section{Zeros of Weighted Bergman Kernels}

Question \textbf{I} recalls the well-known Lu Qi-Keng problem that asks about the zeros of Bergman kernels as the underlying domain changes. A detailed survey of this problem can be found in \cite{Boas98}. 

In this section, after the proof of Theorem \ref{zeros}, we use the Forelli-Rudin formula from \cite{ForelliRudin} (also called inflation principle in \cite{BoasFuStraube99}) to explore zeros of Bergman kernels in higher dimensions.

\begin{proof}[Proof of Theorem \ref{zeros}] Define
\(
  \lambda(r) = \left\{
  \begin{array}{ll}
  18, & 0\leq r \leq \frac{1}{4} \\
  1, & \frac{1}{4}<r\leq1 \\
  \end{array}\right.
  \)\\
For this weight, we explicitly compute the coefficients $\alpha_n$ and get
\begin{equation*}
\alpha_n=\frac{1}{2\pi}\frac{16^{n+1}(2n+2)}{16^{n+1}+17}.
\end{equation*}

\noindent We split the product $(1-z\bar{w})^2B_{\lambda}(z,w)$ into two parts as
\begin{align*}
(1-z\bar{w})^2&\sum_{k=0}^{\infty}\alpha_{k}(z\bar{w})^k=&\\
=&\alpha_{0}+(\alpha_1-2\alpha_0)z\bar{w}+\sum_{k=2}^{\infty}(\alpha_{k}-2\alpha_{k-1}+\alpha_{k-2})(z\bar{w})^k&\\
:=&L(t)+S(t)&\\
\end{align*}
where $L(t)$ denotes the linear part, $S(t)$ denotes the series part and $t$ stands for $z\bar{w}$. Note that $L(t)$ has a zero in $\mathbb{D}$ by explicitly computing $\alpha_0$ and $\alpha_1$. Indeed, $\alpha_0=\frac{16}{33\pi}$ and $\alpha_1=\frac{512}{273\pi}$ so $L(t)$ vanishes at $t=\frac{-91}{170}.$\\

\noindent Next, we show that 
\begin{align}\label{estimate}
\min_{|t|=1-\epsilon}|L(t)|> \max_{|t|=1-\epsilon}|S(t)| 
\end{align}
for small enough $\epsilon>0$.\\

It is clear that $\max_{|t|=1-\epsilon}|S(t)|< \sum_{k=2}^{\infty}|\alpha_{k}-2\alpha_{k-1}+\alpha_{k-2}|$, so it is enough to show
\begin{align}\label{estimate2}
\min_{|t|=1-\epsilon}|L(t)|&>\sum_{k=2}^{\infty}|\alpha_{k}-2\alpha_{k-1}+\alpha_{k-2}|.
\end{align}

By using the explicit formula for $\alpha_k$, we get for $k\geq2$,
\begin{align*}
\alpha_k-\alpha_{k-1}&=\frac{32\times16^{2k}+34\times16^k(15k+16)}{(16^k+17)(16^{k+1}+17)}\\
\alpha_{k-1}-\alpha_{k-2}&=\frac{2\times16^{2k-1}+34\times16^{k-1}(15k+1)}{(16^k+17)(16^{k-1}+17)}.
\end{align*}
By comparing right hand sides, we note that $\alpha_{k}-\alpha_{k-1}<\alpha_{k-1}-\alpha_{k-2}$. Therefore, we get $\alpha_{k}-2\alpha_{k-1}+\alpha_{k-2}<0$ for $k\geq2$ and hence the sum on the right hand side of \eqref{estimate2} is telescoping and converges to $$\alpha_1-\alpha_0-\lim_{k\to\infty}(\alpha_k-\alpha_{k-1})=(\alpha_1-\alpha_0)-2.$$\\

\noindent On the other hand, by a direct calculation, for small enough $\epsilon>0$, 
$$\min_{|t|=1-\epsilon}|L(t)|=(\alpha_1-3\alpha_0)-\epsilon(\alpha_1-2\alpha_0).$$ 
Thus, it remains to show $$(\alpha_1-3\alpha_0)-\epsilon(\alpha_1-2\alpha_0)>(\alpha_1-\alpha_0)-2,$$ for small enough $\epsilon$, in order to get the inequality \eqref{estimate2}. The last inequality is verified by plugging in the actual values of $\alpha$'s and we get the inequality \eqref{estimate}.\\

Once the inequality \eqref{estimate} is obtained, we use Rouche's theorem: $L(t)$ has a zero in $|t|<1-\epsilon$ and $|L(t)|>|S(t)|$ on $|t|=1-\epsilon$  for small enough $\epsilon$; therefore, the sum $L(t)+S(t)=(1-t)^2\sum_{k=0}^{\infty}\alpha_kt^k$ has a zero in $\mathbb{D}$. Since $(1-t)^2$ does not vanish in $\mathbb{D}$ we conclude that $B_{\lambda}(z,w)$ has zeros in $\mathbb{D}\times\mathbb{D}$.\\
\end{proof}

\begin{remark}The same argument can be applied to more weights. We can show that by choosing $A,x>0$ suitably, if we define
\(
  \lambda(r) = \left\{
  \begin{array}{ll}
  A, & 0\leq r \leq x \\
  1, & x<r\leq1 \\
  \end{array}\right.
  \)
then again the weighted Bergman kernel has zeros in $\mathbb{D}\times \mathbb{D}$. 

Furthermore, we can modify these discontinuous functions so that we get smooth radial weight functions for which the weighted Bergman kernels have zeros.\\
\end{remark}

For higher dimensional application, let us take the particular weight $\lambda$ in the proof of Theorem \ref{zeros} and consider the following Reinhardt domain in $\mathbb{C}^2$
\begin{equation*}
\Omega=\{(z,w)\in\mathbb{C}^2: z\in \mathbb{D}~\text{ and } |w|^2<\lambda(z)\}.\\
\end{equation*}

Let $B_{\Omega}\left[(z,w),(t,s)\right]$ denote the unweighted Bergman kernel of the domain $\Omega$ and $B_{\lambda}(z,t)$ be the weighted Bergman kernel of $\mathbb{D}$. We have the following relation between these kernels (see \cite{ForelliRudin}, \cite{Ligocka89}, \cite{BoasFuStraube99}),
\begin{equation}
B_{\Omega}[(z,0),(t,0)]=\frac{1}{\pi}B_{\lambda}(z,t).
\end{equation}

Thus, when the weighted kernel has a zero so does the unweighted kernel in higher dimension.\\

\begin{remark} H.P. Boas observed that Theorem \ref{zeros} can be also proven by using Ramadanov's approximation theorem in \cite{Ramadanov}. A direct computation shows that the weighted Bergman kernel for the weight $(1+k\delta_0)$, where $\delta_0$ is the Dirac mass at $z=0$ and $k>0$, has a zero in $\mathbb{D}\times\mathbb{D}$. We approximate $(1+k\delta_0)$ by smooth functions $\phi_n$ that are all comparable to 1. Then Ramadanov's theorem and Hurwitz's theorem together say that the weighted kernels $B_{\phi_n}(z,w)$ must have zeros in $\mathbb{D}\times \mathbb{D}$ after a certain value of $n$. See \cite{Boas96} for the details of this method.\\
\end{remark}

\section{$L^p$ mapping properties}

Before we prove Theorem \ref{one}, we generalize Question \textbf{II} to the following setup. For a given sequence of complex numbers $\{\beta_n\}$, define a Bergman type integral operator as:
\begin{equation}\label{operator}
Tf(z)=\int_{\mathbb{D}}K(z,w)f(w)\lambda (w)dA(w)~ \text{ where } ~K(z,w)=\sum_{n=0}^{\infty}\beta_n(z\bar w)^n
\end{equation}
and investigate the $L^p$ boundedness of these operators. The following necessary condition is easy to prove.
\begin{proposition}
If the operator $T$, defined in \eqref{operator}, is bounded from $L^p(1)$ to $L^p(1)$ for some $p\in(1, \infty)$
then $\limsup_{n\to \infty}\frac{|\beta_n|}{n}$ is finite.\\ 
\end{proposition}

Unfortunately, it turns out that this necessary condition is not sufficient. Namely, there exists a sequence $\{\gamma_n\}$ such that $\limsup_{n\to \infty}\frac{|\gamma_n|}{n}$ is finite but the associated operator by \eqref{operator} is not bounded from $L^{p_0}(1)$ to $L^{p_0}(1)$ for some $p_0\in (1,\infty)$. See the third section of \cite{Vukotic99}.\\

Schur's lemma is one of the most commonly used arguments to prove boundedness of integral operators and in this section we need this lemma. The proof can be found in \cite{Zhubook}.

\begin{lemma}[Schur]\label{Schur}$(X, A, \mu)$ be a sigma finite measure space and $K(x,y)$ be a positive measurable function on $X\times X$. For $p\in (1,\infty)$ suppose there exists a positive measurable function $h(x)$ on $X$ and a finite number $C>0$ such that 
$$\int_X K(x,y)h(x)^pd\mu(x)\leq C h(y)^p\text { for a.e. }y\in X$$ and
$$\int_X K(x,y)h(y)^qd\mu(y)\leq C h(x)^q\text{ for a.e. }x \in X$$
where $\frac{1}{p}+\frac{1}{q}=1$.\\
Then the operator $Of(y)=\int_X K(x,y)f(x)d\mu(x)$ is bounded on $L^p(X, \mu)$.\\
\end{lemma}

\noindent We prove the next lemma by invoking Schur's lemma.

\begin{lemma}\label{first}
Let $\{\beta_n\}$ be a bounded sequence of complex numbers. Then the operator defined by
\begin{equation}\label{square}
Sf(z)=\int_{\mathbb{D}}\left|\sum_{n=0}^{\infty}\beta_n(z\bar w)^n\right|^2f(w)dA(w)
\end{equation}
is bounded from $L^p(1)$ to $L^p(1)$ for all $p\in(1,\infty)$.
\end{lemma}

\begin{proof} It is enough to show the inequalities in \eqref{Schur} are satisfied with the correct choice of auxiliary function $h(x)$. Particularly, we take $h(w)=(1-|w|^2)^{\epsilon}$ and prove that for any $-1< \epsilon <0$, there exists $C_{\epsilon}>0$ such that for any $z\in \mathbb{D}$
\begin{equation*}
I(\epsilon,z):=\int_{\mathbb{D}}\left|\sum_{n=0}^{\infty}\beta_n(z\bar w)^n\right|^2(1-|w|^2)^{\epsilon}dA(w)\leq C_{\epsilon}(1-|z|^2)^{\epsilon}.
\end{equation*}
We first use the orthogonality  of monomials to get
\begin{align*}
I(\epsilon,z) 
=& \int_{\mathbb{D}}\left(\sum_{n=0}^{\infty}\beta_n(z\bar w)^n\right)\left(\sum_{n=0}^{\infty}\overline{\beta_n}(\bar{z}w)^n\right)(1-|w|^2)^{\epsilon}dA(w)\\
=&\sum_{n=0}^{\infty}|\beta_n|^2|z|^{2n}\int_{\mathbb{D}}|\bar{w}|^{2n}(1-|w|^2)^{\epsilon}dA(w).
\end{align*}
Using boundedness of $\beta_n$'s and adding up the geometric series we obtain
\begin{align*}
I(\epsilon,z)\lesssim&\int_{\mathbb{D}}\sum_{n=0}^{\infty}|z|^{2n}|\bar{w}|^{2n}(1-|w|^2)^{\epsilon}dA(w)\\
=&\int_{\mathbb{D}}\frac{(1-|w|^2)^{\epsilon}dA(w)}{(1-|zw|^2)}\\
=& \pi\int_0^1\frac{(1-r)^{\epsilon}}{(1-|z|^2r)}dr.
\end{align*}
We break up the last integral into two pieces and estimate separately (recall that $\epsilon\in(-1,0)$)
\begin{align*}
I(\epsilon,z)&=\pi\int_0^{|z|^2}\frac{(1-r)^{\epsilon}}{(1-|z|^2r)}dr+\pi\int_{|z|^2}^1\frac{(1-r)^{\epsilon}}{(1-|z|^2r)}dr\\
&\leq \pi\int_0^{|z|^2}\frac{(1-r)^{\epsilon}}{(1-r)}dr+\pi\frac{1}{1-|z|^2}\int_{|z|^2}^1(1-r)^{\epsilon}dr\\
&=\pi\left(\frac{1}{\epsilon}-\frac{(1-|z|^2)^{\epsilon}}{\epsilon}\right)+\pi\frac{(1-|z|^2)^{\epsilon+1}}{(1-|z|^2)(\epsilon+1)}\\
&\leq\frac{-\pi}{\epsilon}(1-|z|^2)^{\epsilon}+\frac{\pi}{\epsilon+1}(1-|z|^2)^{\epsilon}\\
&=\left(\frac{-\pi}{\epsilon}+\frac{\pi}{\epsilon+1}\right)(1-|z|^2)^{\epsilon}.
\end{align*}
Therefore, we conclude that there exists $C_{\epsilon}>0$ such that for any $z\in \mathbb{D}$
\begin{equation*}
I(\epsilon,z)\leq C_{\epsilon}(1-|z|^2)^{\epsilon}.
\end{equation*}

Since the kernel is symmetric, we similarly get the second inequality in \eqref{Schur}. Now for given $p\in(1,\infty)$ choose $\epsilon=\frac{-1}{pq}$ to finish the proof.\\
\end{proof}

By using this lemma, we obtain the following sufficient condition.

\begin{proposition}\label{second}
If $\{\beta_n\}$ is a sequence such that the difference sequence $\{\beta_{n+1}-\beta_n\}$ is bounded then the associated operator $T$, by \eqref{operator}, is bounded from $L^p(1)$ to $L^p(1)$ for all $p\in (1,\infty)$.\\
\end{proposition}
\begin{remark} If $\lambda\equiv 1$ then $\beta_n=\frac{1}{\pi}(n+1)$ so the difference sequence is not only bounded but it is constant $\beta_{n+1}-\beta_{n}=\frac{1}{\pi}$.\\
\end{remark}

\begin{proof} We start with the decomposition of the kernel function $K(z,w)$.
\begin{align*}
K(z,w)=&\sum_{n=0}^{\infty}\beta_n(z\bar{w})^{n}=\frac{1}{1-z\bar{w}}\sum_{n=0}^{\infty}b_n(z\bar{w})^n&\\
=&\sum_{n=0}^{\infty}(z\bar{w})^n\sum_{n=0}^{\infty}b_n(z\bar{w})^n~\text{ for a sequence }\{b_n\}.&
\end{align*}
It is easy to see that $\beta_n=\sum_{k=0}^nb_k$, which implies $b_n=\beta_n-\beta_{n-1}$ (set $\beta_{-1}=0$). Hence, by the assumption in the statement of the proposition, $\{b_n\}$ is a bounded sequence. Furthermore, 
\begin{align*}
|Tf(z)|=&\left|\int_{\mathbb{D}} K(z,w)f(w)dA(w)\right|&\\
\leq&\int_{\mathbb{D}} \left|\sum_{n=0}^{\infty}(z\bar{w})^n\sum_{n=0}^{\infty}b_n(z\bar{w})^n\right||f(w)|dA(w)&\\
\leq&\left(\int_{\mathbb{D}} \left|\sum_{n=0}^{\infty}(z\bar{w})^n\right|^2|f(w)|dA(w)\right)^{\frac{1}{2}}\left(\int_{\mathbb{D}} \left|\sum_{n=0}^{\infty}b_n(z\bar{w})^n\right|^2|f(w)|dA(w)\right)^{\frac{1}{2}}\\
:=&|S_1f(z)|^{\frac{1}{2}}|S_2f(z)|^{\frac{1}{2}}
\end{align*}
where $S_1$ and $S_2$ are respective integral operators.
We integrate both sides with respect to $z$ and apply the Cauchy-Schwarz inequality
\begin{align*}
\int_{\mathbb{D}}|Tf(z)|^pdA(z) \leq &\int_{\mathbb{D}}|S_1f(z)|^{\frac{p}{2}}|S_2f(z)|^{\frac{p}{2}}dA(z)&\\
||Tf||^{2p}_p\leq&||S_1f||^p_p||S_2f||^p_p.\\
\end{align*}
\noindent Note that, $S_1$ and $S_2$ are operators of the type in \eqref{first} and we know such operators are bounded from $L^p(1)$ to $L^p(1)$. Therefore, we get
\begin{equation*}
||Tf||^{2p}_p\leq||S_1f||^p_p||S_2f||^p_p\leq C_1||f||^p_pC_2||f||^p_p\leq C||f||^{2p}_p.
\end{equation*} 
This finishes the proof of Proposition \ref{second}.\\
\end{proof}

\noindent We now give the proof of Theorem \ref{one}.
\begin{proof} We know that $\mathbf{B}_{\lambda}$ is an integral operator of the form \eqref{operator}. Since $\lambda \sim1 $, it is enough to show $\mathbf{B}_{\lambda}$ is bounded from $L^p(1)$ to $L^p(1)$. We use Proposition \ref{second} to do this. Namely, we show that the coefficients $\alpha_n$ of $B_{\lambda}(z,w)$ satisfy the condition in \eqref{second}.

\begin{align*}
\alpha_{n+1}-\alpha_n=&\frac{1}{\int_{\mathbb{D}}|z|^{2n+2}\lambda(z)dA(z)}-\frac{1}{\int_{\mathbb{D}}|z|^{2n}\lambda(z)dA(z)}\\
=&\frac{\int_{\mathbb{D}}(1-|z|^2)|z|^{2n}\lambda(z)dA(z)}{\int_{\mathbb{D}}|z|^{2n+2}\lambda(z)dA(z)\int_D|z|^{2n}\lambda(z)dA(z)}\\
\sim&\frac{\int_{\mathbb{D}}(1-|z|^2)|z|^{2n}dA(z)}{\int_{\mathbb{D}}|z|^{2n+2}dA(z)\int_{\mathbb{D}}|z|^{2n}dA(z)}~\text{ since }\lambda\sim 1\\
\sim & \frac{\int_0^1(1-r^2)r^{2n+1}dr} {\int_0^1r^{2n+3}dr\int_0^1r^{2n+1}dr}~\text{ after switching to polar coordinates}\\
=&(2n+4)(2n+2)\int_0^1(1-r^2)r^{2n+1}dr&\\
=&2&
\end{align*}
We obtain that the sequence $\{\alpha_{n+1}-\alpha_n\}$ is bounded and this finishes the proof of Theorem \ref{one}.\\
\end{proof}
\begin{remark}
An alternative way of proving Theorem \ref{one} is to show that $B_{\lambda}(z,w)$ is a standard kernel in the sense of Coifman-Weiss \cite{CoifmanWeiss71} on the space of homogeneous type $(\mathbb{D},|.|,\lambda)$ where $|.|$ is the Euclidean distance.\\
\end{remark}

\begin{remark}
A generalization and an alternative proof of Theorem \ref{one} appear in \cite{ZeytuncuThesis}.
\end{remark}

\vskip 1cm
\bibliographystyle{plain}
\bibliography{draftNov08}

\begin{thebibliography}{10}

\bibitem{Boas96}
Harold~P. Boas.
\newblock The {L}u {Q}i-{K}eng conjecture fails generically.
\newblock {\em Proc. Amer. Math. Soc.}, 124(7):2021--2027, 1996.

\bibitem{Boas98}
Harold~P. Boas.
\newblock Lu {Q}i-{K}eng's problem.
\newblock {\em J. Korean Math. Soc.}, 37(2):253--267, 2000.
\newblock Several complex variables (Seoul, 1998).

\bibitem{BoasFuStraube99}
Harold~P. Boas, Siqi Fu, and Emil~J. Straube.
\newblock The {B}ergman kernel function: explicit formulas and zeroes.
\newblock {\em Proc. Amer. Math. Soc.}, 127(3):805--811, 1999.

\bibitem{Vukotic99}
Stephen~M. Buckley, Pekka Koskela, and Dragan Vukoti{\'c}.
\newblock Fractional integration, differentiation, and weighted {B}ergman
  spaces.
\newblock {\em Math. Proc. Cambridge Philos. Soc.}, 126(2):369--385, 1999.

\bibitem{CoifmanWeiss71}
Ronald~R. Coifman and Guido Weiss.
\newblock {\em Analyse harmonique non-commutative sur certains espaces
  homog\`enes}.
\newblock Lecture Notes in Mathematics, Vol. 242. Springer-Verlag, Berlin,
  1971.
\newblock {\'E}tude de certaines int{\'e}grales singuli{\`e}res.

\bibitem{Duren04}
Peter Duren and Alexander Schuster.
\newblock {\em Bergman spaces}, volume 100 of {\em Mathematical Surveys and
  Monographs}.
\newblock American Mathematical Society, Providence, RI, 2004.

\bibitem{ForelliRudin}
Frank Forelli and Walter Rudin.
\newblock Projections on spaces of holomorphic functions in balls.
\newblock {\em Indiana Univ. Math. J.}, 24:593--602, 1974/75.

\bibitem{Ligocka89}
Ewa Ligocka.
\newblock On the {F}orelli-{R}udin construction and weighted {B}ergman
  projections.
\newblock {\em Studia Math.}, 94(3):257--272, 1989.

\bibitem{Ramadanov}
I.~Ramadanov.
\newblock Sur une propri\'et\'e de la fonction de {B}ergman.
\newblock {\em C. R. Acad. Bulgare Sci.}, 20:759--762, 1967.

\bibitem{ZeytuncuThesis}
Yunus~E. Zeytuncu.
\newblock ${L}^p$ and {S}obolev {R}egularity of {W}eighted {B}ergman
  {P}rojections.
\newblock {\em Ph.D. Thesis, The Ohio State University}, 2010.

\bibitem{Zhubook}
Kehe Zhu.
\newblock {\em Operator theory in function spaces}, volume 138 of {\em
  Mathematical Surveys and Monographs}.
\newblock American Mathematical Society, Providence, RI, second edition, 2007.

\end{thebibliography}
\vskip .5 cm
\end{document}